\documentclass[10pt]{amsart}
\usepackage{amssymb, amsfonts, latexsym, amsthm, amsmath, eucal, graphics, verbatim, hyperref}

\theoremstyle{plain}
\newtheorem{theorem}{Theorem}[section]

\newtheorem{lemma}[theorem]{Lemma}
\newtheorem{cor}[theorem]{Corollary}
\newtheorem{prop}[theorem]{Proposition}

\newtheorem{conjecture}[theorem]{Conjecture}

\theoremstyle{definition}
\newtheorem{defn}[theorem]{Definition}

\newtheorem{remark}[theorem]{Remark}

\numberwithin{equation}{section}

\newcommand{\ba}{\backslash}

\newcommand{\R}{\mathbb{R}}

\newcommand{\Z}{\mathbb{Z}}

\newcommand{\ov}{\overline}

\newcommand{\val}{\operatorname{val}}
\newcommand{\supp}{\operatorname{Supp}}

\begin{document}
\pagestyle{headings}

\title{Affine Deligne-Lusztig varieties associated to additive affine Weyl group elements}
\author{E. T. Beazley}
\address{Haverford College, Department of Mathematics \& Statistics, 370 Lancaster Avenue, Haverford, PA 19041}
\email{ebeazley@haverford.edu}

\begin{abstract}
Affine Deligne-Lusztig varieties can be thought of as affine analogs of classical Deligne-Lusztig varieties, or Frobenius-twisted analogs of Schubert varieties. We provide a method for proving a non-emptiness statement for affine Deligne-Lusztig varieties inside the affine flag variety associated to affine Weyl group elements satisfying a certain length additivity hypothesis.  In particular, we prove that non-emptiness holds whenever it is conjectured to do 
so for alcoves in the shrunken dominant Weyl chamber, providing a partial converse to the emptiness results of G\"{o}rtz, Haines, Kottwitz, and Reuman. Our technique involves the 
work of Geck and Pfeiffer on cuspidal conjugacy classes, in addition to an analysis of the combinatorics of certain fully commutative elements in the finite Weyl group. 
\end{abstract}

\maketitle

\begin{section}{Introduction}\label{S:intro}

\renewcommand{\thefootnote}{}

Let $G$ be a split connected reductive group over a finite field $k$, and fix a Borel subgroup $B$.  Fix a split maximal torus in $B$, and denote by $W$ the Weyl group of $G$.  Working over a finite field introduces a Frobenius automorphism, denoted by $\sigma$.  In 1976, Deligne and Lusztig constructed a family of algebraic varieties parameterized by the elements of the finite Weyl group $$X_w := \{ g \in G/B \mid g^{-1}\sigma(g) \in BwB \}$$ in order to study the representation theory of finite Chevalley groups (see \cite{DL}, \cite{LuszChev}). Given a fixed element $w \in W$, the associated classical Deligne-Lusztig variety $X_w$ is smooth, equidimensional of dimension $\ell(w)$.  As such, classical Deligne-Lusztig varieties can be thought of as Frobenius-twisted analogs of Schubert varieties.  

More recently, interest in a generalization of classical Deligne-Lusztig varieties to the context of an affine root system has emerged.  This interest was generated in large part because affine Deligne-Lusztig varieties are related to the reduction modulo $p$ of Shimura varieties (see \cite{Hai}, \cite{RapShimura}).  However, in contrast to their classical counterparts, the geometry of affine Deligne-Lusztig varieties is much more difficult to understand.  When we replace $k$ by $k((\pi))$, the corresponding affine variety is parameterized by two elements, one element $x$ from the affine Weyl group, and another element $b$ from the group $G$.  The affine Deligne-Lusztig variety $X_x(b)$ is rarely smooth, and even in the case of $b=1$, the dimension does not usually coincide with $\ell(x)$. Furthermore, unlike in the classical case in which Lang's Theorem holds so that $X_w$ is always non-empty, the affine Deligne-Lusztig variety $X_x(b)$ is frequently empty.  

Until quite recently the question of determining for which pairs $(x,b)$ the associated affine Deligne-Lusztig variety is non-empty as a set has remained largely open.  In the context of affine Deligne-Lusztig varieties inside the affine Grassmannian, the non-emptiness question has been settled, and a dimension formula has also been proved in \cite{GHKR} and \cite{VieDim}.  Here, the characterization for non-emptiness is phrased in terms of Mazur's inequality, which is a group-theoretic generalization of the inequality between the Hodge and Newton vectors in crystalline cohomology (see \cite{Kat}, \cite{Maz}).  Mazur's inequality relates the cocharacter $\mu$ from the torus part of the affine Weyl group element $x$ and the Newton polygon of $b$.  If $\nu$ denotes the Newton polygon associated to $b$, Mazur's inequality roughly states that $\nu \leq \mu$, which means that the difference $\mu-\nu$ is a non-negative linear combination of positive coroots.  That Mazur's inequality is necessary for non-emptiness was proved by Kottwitz and Rapoport in \cite{KRFcrystals}, in which they also proved that this criterion is sufficient for $G=GL_n$ and $GSp_{2n}$ in the context of the affine Grassmannian.  Lucarelli then proved that Mazur's inequality guarantees non-emptiness for the classical groups in \cite{Luc}, and Gashi settled the question for the exceptional groups in \cite{Gas} and \cite{GasGLn}.

The current paper is one of several recent papers which addresses the non-emptiness question for affine Deligne-Lusztig varieties inside the affine flag variety.  In the context of the affine flag variety, there is no simple inequality that characterizes the non-emptiness pattern.  Here, Mazur's inequality is again necessary, but for most pairs $(x,b)$ this inequality does not suffice to yield non-emptiness of $X_x(b)$.  In the simplest situation in which $b=1$ and $x$ corresponds to an alcove in the ``shrunken'' Weyl chambers, Reuman characterized the elements $x$ for which $X_x(1)$ is non-empty in types $A_1, A_2,$ and $C_2$ (see \cite{Reu}), and he conjectured that his findings were a general pattern.  In \cite{Be1}, we extended Reuman's results to classify all pairs $(x,b)$ for which $X_x(b)$ is non-empty in the case $G=SL_3$.  A recent result of G\"{o}rtz, Haines, Kottwitz, and Reuman \cite{GHKRadlvs} provides a
description, which is conjectured to be a characterization, for determining emptiness of affine Deligne-Lusztig
sets associated to any basic loop group element, and our result provides a partial converse to the main theorem in \cite{GHKRadlvs}. 

In this paper, we use purely combinatorial techniques to prove non-emptiness for a class of affine Deligne-Lusztig varieties associated to affine Weyl group elements which satisfy a certain length additivity criterion.  The two main ingredients are the theory of cuspidal conjugacy classes as developed by Geck and Pfeiffer in \cite{GP}, and the combinatorics of certain fully commutative elements in the finite Weyl group (see \cite{StFC}).  The fully commutative elements in a Coxeter group have special properties which are related to the smoothness of Schubert varieties.  In particular, the characterization of fully commutative elements in terms of pattern avoidance given by Stembridge in \cite{StRW} is used by Billey and Postnikov in \cite{BillPost} to generalize the results of Lakshmibai and Sandhya in \cite{LakSand} and Fan in \cite{Fan} to describe families of smooth Schubert varieties.  It would be interesting to see if this connection generalizes to the Frobenius-twisted context to provide additional geometric information about affine Deligne-Lusztig varieties.

We also mention several other recent independent results on the non-emptiness question inside the affine flag variety.  In \cite{HeWC}, He proves a non-emptiness pattern for $X_x(1)$ if the translation part of $x$ is quasi-regular.  Most recently, G\"{o}rtz and He prove the non-emptiness conjecture in \cite{GHKRadlvs}, although still under Reuman's original shrunken hypothesis.  It remains a hard problem to characterize non-emptiness for alcoves which lie outside the shrunken Weyl chambers, and the only insights into this problem to date occur in the work of Reuman \cite{Reu} and the author \cite{Be1} for groups of low rank.

\begin{subsection}{Notation}\label{S:notation}

Let $k$ be a finite field with $q$ elements, and let $\overline{k}$ be an algebraic closure of $k$.  Denote by $\pi$ the uniformizing element of the discrete valuation ring $\mathcal{O}:= \ov{k}[[\pi]]$, having fraction field $L:=\overline{k}((\pi))$ and maximal ideal $P:= \pi \mathcal{O}$.  Normalize the valuation homomorphism $\val: L^{\times} \rightarrow \Z$ so that $\val(\pi) = 1$.  We can extend the usual Frobenius automorphism $x \mapsto x^q$ on $\ov{k}$ to a map $\sigma: L \rightarrow L$ given by $\sum a_i\pi^i \mapsto \sum a_i^q\pi^i$.  Denote by $F:= k((\pi))$ the Frobenius fixed subfield of $L$.

Let $G$ be a split connected reductive group over $k$, and let $B$ denote a fixed Borel subgroup and $A$ a
maximal torus in $B$. Let $W$ denote the Weyl group of $A$ in $G$.  Let $S$ be the set of finite simple reflections, and $W_T$ the subgroup of $W$ generated by $T \subset S$. For a fixed $w \in W$, denote by $D_L(w)$ and $D_R(w)$ the left and right descent sets of $w$, respectively; \textit{i.e.}, $D_L(w) = \{ s \in S \mid \ell(sw) < \ell(w) \}$ and analogously for $D_R(w)$. The descent set of $w$ will be denoted $D(w) = D_L(w) \cup D_R(w)$.  Let $\supp(w)$ denote the support of $w$, which is the set of simple reflections used in any (equivalently every) reduced expression for $w$.  If $\supp(w) = S$ we say that $w$ is of full support.

Let $P = MN$ be a parabolic subgroup of $G$
that contains $B$, where $M$ is the unique Levi subgroup of $P$ that contains $A$ and $N$ is the unipotent radical.
Denote by $R$ the set of roots of $A$ in $G$ and by $R^+$ and $R^-$ the corresponding set of positive and
negative roots, respectively.  Denote by $\rho$ the half-sum of the positive roots of $A$.  For a cocharacter
$\lambda \in X_*(A)$, we will write $\pi^{\lambda}$ for the element in $A(L)$ that is the image of $\pi$ under
the homomorphism $\lambda: \mathbb{G}_m \rightarrow A$.  Denote by $\Lambda_G$ the quotient of $X_*(A)$ by the coroot lattice, and by $\eta_G$ the surjection $\eta_G: G(L) \rightarrow \Lambda_G$.

Let $\mathfrak{a} := X_*(A)_{\R}$. The dominant Weyl chamber $C$ is defined to be the set of $x \in \mathfrak{a}$ such that $\langle \alpha, x \rangle > 0$ for every $\alpha \in R^+$. Analogously, denote by $C^0$ the antidominant Weyl chamber, or the set of $x \in \mathfrak{a}$ such that $\langle \alpha, x \rangle < 0$ for all $\alpha \in R^+$.  Our convention will be to call the unique alcove in $C$ whose
closure contains the origin the base alcove $\mathbf{a}$. Let $I$ be the associated Iwahori subgroup of the loop group $G(L)$, which fixes the base alcove $\mathbf{a}$.  The Iwahori subgroup $I$ is the inverse image of the opposite Borel subgroup under the projection map $G(\mathcal{O}) \rightarrow G(\overline{k})$.

Denote by $\widetilde{W} = X_*(A)\rtimes W$ the (extended) affine Weyl group, which acts on the set of (extended) alcoves in $\mathfrak{a}$. We shall usually express an element $x\in
\widetilde{W}$ as $x = \pi^{\mu}w$, where $\mu \in X_*(A)$ and $w \in W$. An alcove in $\mathfrak{a}$ can be written as $x\mathbf{a}$ for a unique $x \in \widetilde{W}$, and we will frequently use this correspondence between alcoves and elements in the affine Weyl group element without comment.

\end{subsection}

\begin{subsection}{Affine Deligne-Lusztig varieties}\label{S:adlvs}

 Let $x \in \widetilde{W}$ and $b \in G(L)$.  The affine Deligne-Lusztig variety associated to the elements $x$ and $b$, denoted $X_x(b)$, inside the affine flag manifold is defined as follows: \begin{equation*}
 X_x(b) := \{ g \in G(L)/I : g^{-1}b\sigma(g) \in IxI\}.
\end{equation*}

We now recall one version of the conjecture for non-emptiness of these affine Deligne-Lusztig varieties in the case that $b$ is basic.
Denote by $\eta_1: \widetilde{W} \rightarrow W$ the surjection from the affine Weyl group onto the finite Weyl
group, and by $\eta_2: \widetilde{W} \rightarrow W$ the map that associates to each alcove the finite Weyl
chamber in which it lies.
\begin{conjecture}[G\"{o}rtz-Haines-Kottwitz-Reuman]\label{T:reuman}
Suppose that $x$ lies in the shrunken Weyl chambers, and let $b \in G(L)$ be basic.  If $\eta_G(x) = \eta_G(b)$ and \begin{equation}\label{E:reuman} \eta_2^{-1}(x)\eta_1(x)\eta_2(x) \in W \ba \bigcup_{T \subsetneq S} W_T,
\end{equation} then $X_x(b) \neq \emptyset$.
\end{conjecture} \noindent  Informally, the shrunken Weyl chambers are a union of alcoves which do not lie too close to the walls of the Weyl chambers.  We refer the reader to \cite{Reu} for Reuman's original definition of the shrunken Weyl chambers, and to \cite{GHKRadlvs} for further discussion on this conjecture.

The above stated version of the non-emptiness conjecture is the one with which we shall work, although we point out that G\"{o}rtz, Haines, Kottwitz, and Reuman use the notion of $P$-alcoves to extend this conjecture to include the ``non-shrunken'' Weyl chambers (see Conjecture 1.1.1 in \cite{GHKRadlvs}).  We prefer the above formulation of the non-emptiness conjecture for the purposes of this paper because it is more convenient to characterize
elements $x\in \widetilde{W}$ for which we expect non-emptiness using a combinatorial description derived easily
from criterion \eqref{E:reuman}.  In particular, the set $W \ba \bigcup_{T \subsetneq S} W_T$ consists of all elements
of $W$ such that any reduced expression contains all simple reflections; \textit{i.e.} the finite Weyl group elements that have full support.

\end{subsection}

\begin{subsection}{Statement of the theorem}\label{S:thm}

Write $x = \pi^{\mu}w$ and note that for $\mu$ dominant, condition \eqref{E:reuman} reduces to $w \in W \ba
\bigcup_{T \subsetneq S} W_T$. In this paper, we replace the condition on the conjugate of $w$ appearing in \eqref{E:reuman} with two different hypotheses.  Namely, we will require $w$ itself to be of full support, and that the product $\eta_2^{-1}(x)\eta_1(x)$ is length additive.

As such, we shall often restrict ourselves to considering elements $x = \pi^{\mu}w$ such that
any reduced expression for $w$ contains all simple reflections.  It will be useful to have more descriptive terminology to refer to this simplified version of condition \eqref{E:reuman} on an affine Weyl group element, which we now introduce.
\begin{defn}
Let $x = \pi^{\mu}w \in \widetilde{W}$, where $w \in W$.  If $\supp(w) = S$, we say that $x$ is \textit{full}.
\end{defn}
\noindent Observe that $x \in \widetilde{W}$ being full is solely a condition on the finite part of $x$.  In addition, the reader should note that $x$ full is only sufficient for non-emptiness in the case of $\mu$ dominant.  In general one must consider the conjugate of $w$ specified by \eqref{E:reuman}. 

Now write $x = \pi^{v(\mu)}w\in \widetilde{W}$ where $\mu$ is dominant and $w,v \in W$.  One additional hypothesis on $x$ in our main theorem requires that the length of the product $v^{-1}w$ equals the sum of the lengths of $v$ and $w$.  For ease of reference, we introduced the following terminology to describe this length additivity criterion.
\begin{defn}
Let $x = \pi^{v(\mu)}w \in \widetilde{W}$, where $\mu$ is dominant and $w,v\in W$.  If $\ell(v^{-1}w) = \ell(v^{-1}) + \ell(w)$, then we say that $x$ is \emph{additive}.
\end{defn} 

In this paper we treat the situation in which $\mu$ is also regular, which permits use of various formulas for the length of affine Weyl group elements in Section \ref{S:length}, in addition to keeping our elements inside the shrunken Weyl chambers. Recall that for $G = GL_n$, the cocharacter $\mu$ is regular if and only if $\mu = (\mu_1, \mu_2, \dots, \mu_n)$
satisfies $\mu_1 > \mu_2 > \cdots > \mu_n$. 

Our goal is to prove Conjecture \ref{T:reuman} for a particular class of affine Weyl group elements, providing a partial converse to the main theorem of G\"{o}rtz, Haines, Kottwitz, and Reuman in \cite{GHKRadlvs}.  In particular, we prove this conjecture in the case in which $x = \pi^{v(\mu)}w \in \widetilde{W}$ is full and additive, and $\mu$ is regular dominant.

\begin{theorem}\label{T:main}
Suppose $G=GL_n$, or that $G$ is of type $C_2,$ or $G_2$.  Let $x=\pi^{v(\mu)}w \in \widetilde{W}$, where $v,w \in W$ are such that $\ell(v^{-1}w) = \ell(v^{-1})+\ell(w)$, and $\mu$ is regular dominant.  Let $b\in G(L)$ be basic.  Then if $\eta_G(x) = \eta_G(b)$ and $w\in W \ba
\bigcup_{T \subsetneq S} W_T$, we have $X_x(b) \neq \emptyset$.
\end{theorem}

We should mention that the majority of the results used in the course of the proof of Theorem \ref{T:main} are actually type-free.  The only statements which involve a specialization to type $A_n$ occur in Section \ref{S:except}, in which we carefully provide formal remarks discussing both the possibilities for and obstructions to generalizing any such statements to other types. 

As an obvious corollary to this theorem, we prove the conjectured non-emptiness criterion for elements $x=\pi^{\mu}w$ such that $\mu$ is regular dominant, since the requisite length hypothesis $\ell(v^{-1}w) = \ell(v^{-1})+\ell(w)$ clearly holds for $v = 1$, and the condition that $x$ is full coincides with the non-emptiness criterion \eqref{E:reuman} in this case.  In general, there are many affine Weyl group elements that satisfy the additivity and fullness criteria specified in Theorem \ref{T:main}, although the dominant Weyl chamber is the only Weyl chamber where all shrunken alcoves for which non-emptiness is predicted satisfy these hypotheses.  

\begin{cor}\label{T:maincor}
Suppose $G=GL_n$, or that $G$ is of type $C_2,$ or $G_2$.  Let $x=\pi^{\mu}w \in \widetilde{W}$ be such that $\mu$ is regular dominant, and let $b\in G(L)$ be basic.  Then if $\eta_G(x) = \eta_G(b)$ and $w\in W \ba
\bigcup_{T \subsetneq S} W_T$, we have $X_x(b) \neq \emptyset$. 
\end{cor}

As previously mentioned, G\"{o}rtz and He have recently obtained independent results which prove a more general non-emptiness result in \cite{GHdim}.  Both papers use generalizations of the geometric results of Deligne and Lusztig (see Section \ref{S:nonemptylemmas}) to construct an inductive proof of non-emptiness.  However, the particular combinatorial arguments involved in building the inductive process differ somewhat significantly, and both arguments use a variety of combinatorial results on affine and/or finite Weyl groups that may themselves be of independent interest.  

The primary distinguishing feature of our proof involves a detailed study in Section \ref{S:except} of the combinatorics of certain Coxeter elements which arise naturally in the context of analyzing finite Weyl group elements of full support which lose the full support condition when multiplied by any simple reflection in their descent set. An additional distinctive benefit of our argument is that, given $x \in \widetilde{W}$ satisfying the hypotheses of Theorem \ref{T:main}, our Proposition \ref{T:swsprop} provides an explicit algorithm for constructing a sequence of affine Weyl group elements $x_e, x_{e-1}, \dots, x_1, x_0=x$ such that $x_e$ is elliptic and $$X_{x_e}(b) \neq \emptyset \Longrightarrow \cdots \Longrightarrow X_{x_1}(b) \neq \emptyset \Longrightarrow X_x(b) \neq \emptyset.$$  The argument of G\"{o}rtz and He inductively asserts only the existence of such an element $x_e$, albeit for a more general class of affine Weyl group elements.

\vspace{5pt}
\noindent \textbf{Acknowledgments.}
The author would like to thank Robert Kottwitz for the idea to structure an argument based upon results in the original paper of Deligne and Lusztig, for helpful conversations while working on this project, and for a careful reading of several drafts of this paper.  Ulrich G\"{o}rtz also provided several useful comments on an earlier preprint.  We additionally thank John Stembridge for many beneficial discussions regarding the details of several of the combinatorial results.

\end{subsection}

\end{section}

\begin{section}{Several non-emptiness results}\label{S:nonemptylemmas}

\begin{subsection}{Non-emptiness resulting from geometric structure}\label{S:bundle}

We proceed by stating several results which yield reduction steps.  As in Deligne and Lusztig's original paper
\cite{DL}, under certain hypotheses on the length of $x \in \widetilde{W}$, the variety $X_x(b)$ can be written as the disjoint union of two bundles over affine
Deligne-Lusztig varieties associated to elements smaller than $x$ in the partial ordering on $\widetilde{W}$.
This description allows us to prove non-emptiness of $X_x(b)$ by showing non-emptiness for certain $X_{x'}(b)$
where $x' < x$.  As before, we denote by $S$ the set of finite simple reflections.

\begin{prop}\label{T:bundle}
Let $x \in \widetilde{W}$ and $s \in S$, and suppose that $\ell(x) > \ell(sx) = \ell(xs) > \ell(sxs)$.  For $b \in G(L)$, if any of $X_{sx}(b)$, $X_{xs}(b)$, or $X_{sxs}(b)$ is non-empty, then $X_{x}(b)$ is also non-empty.
\end{prop}

In addition, for the affine Weyl group elements in Proposition \ref{T:bundle} which are of the same length, the associated affine Deligne-Lusztig varieties are in bijection as sets.  Equivalently, replacing $x$ with either $sx$ or $xs$ in the hypotheses of Proposition \ref{T:bundle} also gives a bijection between the affine Deligne-Lusztig varieties associated to $x$ and $sxs$.  These facts indicate that, for the purposes of this paper, we are free to work with either element we choose.

\begin{prop}\label{T:isom}
Let $x \in \widetilde{W}$ and $s \in S$, and suppose that either $\ell(sx) > \ell(x) = \ell(sxs) > \ell(xs)$ or $\ell(xs) > \ell(x) = \ell(sxs) > \ell(sx)$.  For $b \in G(L)$, we then have $X_{x}(b) \neq \emptyset \iff X_{sxs}(b) \neq \emptyset$.
\end{prop}

For the proofs of these two non-emptiness statements and further discussion on the reduction method of Deligne and Lusztig, we refer to Section 2.5 in \cite{GHdim}.

\end{subsection}

\begin{subsection}{Non-emptiness from elliptic elements}\label{S:elliptic}

The reason that Proposition \ref{T:bundle} is useful is that it will permit us to construct an inductive argument for proving non-emptiness of the affine Deligne-Lusztig varieties of interest, where the induction is done on the length of the finite part of $x$.  Our goal will be to construct an argument that repeatedly applies Propositions \ref{T:bundle} and \ref{T:isom}, until we are reduced to proving that $X_x(b) \neq \emptyset$ for an element $x$ such that non-emptiness of the associated affine Deligne-Lusztig variety is already known.  
\begin{defn} An element $w \in W$ is called \textit{elliptic} if it is not contained in any conjugate of a proper parabolic subgroup of $W$.
\end{defn}  
\noindent We may occasionally abuse language and refer to an element $x = \pi^{\mu}w \in \widetilde{W}$ as elliptic, by which we mean that its finite part $w$ is elliptic.  The elliptic conjugacy classes have been extensively studied and can be characterized in many ways, (see \cite{GPbook}, for example, where these conjugacy classes are referred to as \textit{cuspidal}).  There are many equivalent definitions of elliptic elements in any Weyl group, but we have chosen to recall the version that is most convenient for type $A_n$ considerations.

The base case for our inductive argument will be the class of elliptic elements.  The following proposition appears as Lemma 9.4.3 in \cite{GHKRadlvs}, so we refer the reader there for the proof.

\begin{prop}\label{T:elliptic}
Let $x \in \widetilde{W}$ be elliptic and $b \in G(L)$ basic.  Then $$X_x(b) \neq \emptyset \iff \eta_G(x) = \eta_G(b).$$
\end{prop}

\noindent Note that the definition of elliptic, and therefore the previous proposition, does not depend on the Weyl chamber in which $x$ lies.  We further discuss elliptic conjugacy classes and their properties in Section \ref{S:conj}.

\end{subsection}

\end{section}

\begin{section}{Reduction to the finite case}\label{S:length}

We now present the key proposition that makes the proof of the main theorem work.  The general strategy will be to construct an inductive process, wherein Proposition \ref{T:swsprop} is applied successively until eventually obtaining an elliptic element.  In particular, the reader will note that in the statement of each of the five cases of the proposition, the affine Weyl group element that should replace $x$ is explicitly provided and the requisite length relationship to reapply the proposition to the new element is verified.  

\begin{prop}\label{T:swsprop}
Let $x=\pi^{v(\mu)}w \in \widetilde{W}$, where $v,w \in W$ are such that $\ell(v^{-1}w) = \ell(v^{-1})+\ell(w)$, and $\mu$ is regular dominant.  Suppose that $\ell(sws) \leq \ell(w)$ and $w \neq sws$.
\begin{enumerate} 
\item Suppose $\ell(sws) < \ell(w)$.  If $X_{sx}(b) \neq \emptyset$, then $X_x(b) \neq \emptyset$.  In this case, $sx=\pi^{sv(\mu)}sw$ and $\ell(v^{-1}s\cdot sw) = \ell(v^{-1}s)+\ell(sw)$.
\item Suppose $\ell(sws) = \ell(w)$ and $s \in D_L(w)$.
\begin{enumerate}
\item If $\ell(v^{-1}ws) = \ell(v^{-1}w)+1$ and $X_{sxs}(b) \neq \emptyset$, then $X_x(b) \neq \emptyset$.  In this case, $sxs=\pi^{sv(\mu)}sws$ and $\ell(v^{-1}s\cdot sws) = \ell(v^{-1}s)+\ell(sws)$.
\item If $\ell(v^{-1}ws) = \ell(v^{-1}w)-1$ and $X_{sx}(b) \neq \emptyset$, then $X_x(b) \neq \emptyset$.  In this case, $sx=\pi^{sv(\mu)}sw$ and $\ell(v^{-1}s\cdot sw) = \ell(v^{-1}s)+\ell(sw)$.
\end{enumerate}
\item Suppose $\ell(sws) = \ell(w)$ and $s \in D_R(w)$.
\begin{enumerate}
\item If $\ell(v^{-1}s) = \ell(v^{-1})-1$ and $X_{sxs}(b) \neq \emptyset$, then $X_x(b) \neq \emptyset$.  In this case, $sxs=\pi^{sv(\mu)}sws$ and $\ell(v^{-1}s\cdot sws) = \ell(v^{-1}s)+\ell(sws)$.
\item If $\ell(v^{-1}s) = \ell(v^{-1})+1$ and $X_{xs}(b) \neq \emptyset$, then $X_x(b) \neq \emptyset$.  In this case, $xs=\pi^{v(\mu)}ws$ and $\ell(v^{-1}\cdot ws) = \ell(v^{-1})+\ell(ws)$.
\end{enumerate}
\end{enumerate}
\end{prop}

The main idea in the proof of Proposition \ref{T:swsprop} is to prove that the requisite length relationships hold in order to apply Propositions \ref{T:bundle} and \ref{T:isom}.  The reader will observe that the hypotheses appearing in Proposition \ref{T:swsprop} are solely on the finite Weyl group part, eliminating the need to verify conditions on length relationships inside the affine Weyl group like the ones we see in Propositions \ref{T:bundle} and \ref{T:isom}.  We begin by recalling a formula for the length of affine Weyl group elements (see \cite{HKP}), which then gives rise to a useful reformulation.

\begin{prop}\label{T:lengthform}
Let $\mu \in X_*(A)$ be regular dominant, and let $w_1, w_2 \in W$.  Then \begin{equation}
\ell(w_1\pi^{\mu}w_2) = \ell(w_2) + \ell(\pi^{\mu}) - \ell(w_1). \end{equation}
\end{prop}

\begin{lemma}\label{T:lengthobs} Let $\mu \in X_*(A)$ be regular dominant, and let $w_1, w_2 \in W$.  Then \begin{equation}\label{T:obseq}\ell(\pi^{w_1^{-1}(\mu)}w_2) = \ell(\pi^{\mu}w_1w_2) - \ell(w_1).\end{equation}
\end{lemma}

\begin{proof}
Rewrite the expression
\begin{align}
\pi^{w_1^{-1}(\mu)}w_2 &= w_1^{-1}w_1\pi^{w_1^{-1}(\mu)}w_2 \notag \\
\phantom{\pi^{w_1^{-1}(\mu)}w_2} &= w_1^{-1}\pi^{w_1w_1^{-1}(\mu)}w_1w_2 \notag \\
\phantom{\pi^{w_1^{-1}(\mu)}w_2} &= w_1^{-1}\pi^{\mu}w_1w_2. \notag
\end{align}
Therefore, Proposition \ref{T:lengthform} says that 
\begin{align}
\ell(\pi^{w_1^{-1}(\mu)}w_2) &= \ell(w_1^{-1}\pi^{\mu}w_1w_2) \notag \\
\phantom{\ell(\pi^{w_1^{-1}(\mu)}w_2)} &= \ell(w_1w_2) + \ell(\pi^{\mu}) - \ell(w_1^{-1}) \notag \\
\phantom{\ell(\pi^{w_1^{-1}(\mu)}w_2)} &= \ell(\pi^{\mu}w_1w_2) - \ell(w_1). \notag
\end{align} Here to obtain the final equality we have used that $\mu$ is dominant and that $\ell(w_1) = \ell(w_1^{-1})$.
\end{proof}

\begin{proof}[Proof of Proposition \ref{T:swsprop}]

We begin by computing formulas for the lengths of $x, sx, xs,$ and $sxs$.   Compute using Lemma \ref{T:lengthobs} and the fact that $\mu$ is dominant that \begin{align}\ell(x) &= \ell(\pi^{\mu}) +\ell(v^{-1}w) - \ell(v^{-1}) \label{E:xlength} \\ \ell(sx) &= \ell(\pi^{\mu}) +\ell(v^{-1}w) - \ell(v^{-1}s) \label{E:sxlength} \\  \ell(xs) &= \ell(\pi^{\mu}) +\ell(v^{-1}ws) - \ell(v^{-1}) \label{E:xslength}\\  \ell(sxs) &= \ell(\pi^{\mu}) +\ell(v^{-1}ws) - \ell(v^{-1}s) \label{E:sxslength} \end{align} 

We now analyze each of the cases in the statement of the proposition separately, although the structure of all five arguments is quite similar.

\vskip 5 pt
\textit{Case (1):}
\vskip 5 pt

We argue that, in this case, we have $\ell(x) > \ell(sx) = \ell(xs) > \ell(sxs)$, in which case we apply Proposition \ref{T:bundle} to yield the non-emptiness result.  First note that if $\ell(sws) < \ell(w)$ and $w \neq sws$, then we necessarily have that $\ell(w) > \ell(sw) = \ell(ws) > \ell(sws)$.  Comparing equations \eqref{E:xlength} and \eqref{E:xslength}, and \eqref{E:sxlength} and \eqref{E:sxslength}, we see that we always have that $\ell(x) > \ell(xs)$ and $\ell(sx)>\ell(sxs)$ in this case.  Further, by hypothesis we have $\ell(v^{-1}w) = \ell(v^{-1}) + \ell(w)$, which necessarily implies that $\ell(v^{-1}s) = \ell(v^{-1})+1$, since $\ell(sw) < \ell(w)$.  Comparing equations \eqref{E:xlength} and \eqref{E:sxlength} we therefore see that $\ell(x) > \ell(sx)$.  The only way for the three inequalities $\ell(x) > \ell(xs)$, $\ell(sx)>\ell(sxs)$, and $\ell(x) > \ell(sx)$ to simultaneously hold is to have $$\ell(x) > \ell(sx) = \ell(xs) > \ell(sxs).$$ Proposition \ref{T:bundle} thus applies and says that if $X_{sx}(b) \neq \emptyset$, then $X_x(b) \neq \emptyset.$

Finally, consider $\ell(v^{-1}s\cdot sws)$, recalling that $\ell(v^{-1}s) = \ell(v^{-1}) +1$ in this case.  We compute
\begin{align*} \ell(v^{-1}s\cdot sws) &=  \ell(v^{-1}ws) \\ \phantom{\ell(v^{-1}s\cdot sws)} &= \ell(v^{-1}w) - 1 \\  \phantom{\ell(v^{-1}s\cdot sws)} &= \ell(v^{-1}) + 1 + \ell(w)-2 \\  \phantom{\ell(v^{-1}s\cdot sws)} &= \ell(v^{-1}s) + \ell(sws). \end{align*}

\textit{Case (2):}
\vskip 5 pt

First note that $\ell(ws) > \ell(w) = \ell(sws) > \ell(sw)$, since $w \neq sws$, $\ell(sws) = \ell(w)$, and $s\in D_L(w)$ in this case. Further, since $\ell(v^{-1}w) = \ell(v^{-1}) + \ell(w)$ and $s \in D_L(w)$, we also automatically have  $\ell(v^{-1}s) = \ell(v^{-1})+1$. These observations imply that we always have $\ell(x) > \ell(sx)$ and $\ell(xs)>\ell(sxs)$ in Case (2).  Cases (2a) and (2b) are then completely determined by the value of $\ell(v^{-1}ws) = \ell(v^{-1}w) \pm 1$ and thus whether we have $\ell(xs) > \ell(x)$ or $\ell(x) > \ell(xs)$, respectively.

\vskip 5pt
\textit{Case (2a):}
\vskip 5 pt

In this case, $\ell(v^{-1}ws) = \ell(v^{-1}w)+1$, whence $\ell(xs) > \ell(x)$ and so $$\ell(xs) > \ell(x) =\ell(sxs) > \ell(sx).$$  Proposition \ref{T:isom} then tells us that $X_x(b) \neq \emptyset$ if and only if $X_{sxs}(b) \neq \emptyset$.  Finally, we see that
\begin{align*} \ell(v^{-1}s\cdot sws) &=  \ell(v^{-1}ws) \\ \phantom{\ell(v^{-1}s\cdot sws)} &= \ell(v^{-1}w) + 1 \\  \phantom{\ell(v^{-1}s\cdot sws)} &= \ell(v^{-1}) + 1 + \ell(w)\\  \phantom{\ell(v^{-1}s\cdot sws)} &= \ell(v^{-1}s) + \ell(w) \\  \phantom{\ell(v^{-1}s\cdot sws)} &= \ell(v^{-1}s) + \ell(sws). \end{align*}

\vskip 5pt
\textit{Case (2b):}
\vskip 5 pt

In this case, $\ell(v^{-1}ws) = \ell(v^{-1}w)-1$, and so $\ell(x) > \ell(xs)$.  Consequently, $$\ell(x) > \ell(sx) = \ell(xs) > \ell(sxs).$$ Proposition \ref{T:bundle} then tells us that if $X_{sx}(b) \neq \emptyset$, then $X_x(b) \neq \emptyset$.  Finally, we compute that
\begin{align*} \ell(v^{-1}s\cdot sw) &=  \ell(v^{-1}w) \\ \phantom{\ell(v^{-1}s\cdot sw)} &= \ell(v^{-1})+\ell(w)  \\  \phantom{\ell(v^{-1}s\cdot sw)} &= \ell(v^{-1}) + 1 + \ell(w)-1 \\  \phantom{\ell(v^{-1}s\cdot sw)} &= \ell(v^{-1}s) + \ell(sw). \end{align*}

The proof of Case (3) is proceeds in the same manner as Case (2), and so we leave the details to the reader.
\end{proof}

\end{section}

\begin{section}{Elliptic conjugacy classes}\label{S:conj}

The choice of affine Weyl group element which replaces $x$ in the inductive argument established by Proposition \ref{T:swsprop}, say $x'$, needs to possess several key features.  Clearly, it should be true that $X_{x'}(b) \neq \emptyset$ implies $X_x(b) \neq \emptyset$, and the element $x'$ should again satisfy the hypotheses of Proposition \ref{T:swsprop}, both of which are demonstrated in the statement of the proposition. In addition, we should expect to be able to eventually obtain an elliptic element after applying Proposition \ref{T:swsprop}, either to $x'$ or its successors.   As such, the proof of Theorem \ref{T:main} requires an understanding of elliptic conjugacy classes, since the goal of the argument will be to eventually apply Proposition \ref{T:elliptic}.  Chapter 3 in \cite{GPbook} is devoted to the study of elliptic conjugacy classes in finite Weyl groups, and we recall in this section several results of Geck and Pfeiffer that are essential to our ability to inductively apply Proposition \ref{T:swsprop}.

\begin{subsection}{Minimal length in conjugacy classes}

The main theorem in \cite{GP} guarantees the existence of the simple reflection $s$ in the statement of Proposition \ref{T:swsprop} at each stage.  In particular, Geck and Pfeiffer provide a means for taking an element $w \in W$ and conjugating by a sequence of simple reflections to obtain an element of minimal length inside that conjugacy class in $W$.  We review the notation and the statement of the main theorem in \cite{GP} for the sake of completeness.

Given $w, w' \in W$ and $s \in S$, write $w \xrightarrow{s} w'$ if $w' = sws$ and $\ell(w') \leq \ell(w)$.  If we have a sequence of elements $w = w_1, \dots, w_n = w'$ such that for every $i=2, \dots, n$ we have $w_{i-1} \xrightarrow{s_i} w_i$ for some $s_i \in S$, then write $w \rightarrow w'$.  For a conjugacy class $\mathcal{C}$ of $W$, we denote by $\mathcal{C}_{min}$ the set of elements in $\mathcal{C}$ that have minimal length.

\begin{theorem}[Geck-Pfeiffer]\label{T:sws}
Let $\mathcal{C}$ be a conjugacy class of $W$.  Then for each $w \in \mathcal{C}$, there exists a $w' \in \mathcal{C}_{min}$ such that $w \rightarrow w'$.
\end{theorem}

We point out that this result of Geck and Pfeiffer holds for any finite Weyl group $W$, including the exceptional groups.  We also remark that this result for finite Weyl groups, rather than affine Weyl groups, will be sufficient for our purposes in light of Proposition \ref{T:swsprop}.

Given an affine Weyl group element $x=\pi^{v(\mu)}w$, Theorem \ref{T:sws} guarantees that there exists an element $w' \in \mathcal{C}_{min}$ such that $w \rightarrow w'$.  If $\mu$ is regular dominant and $\ell(v^{-1}w)=\ell(v^{-1}) + \ell(w)$, we may, roughly speaking, apply Proposition \ref{T:swsprop} to obtain an affine Weyl group element $x'$ such that the finite part has decreased in length.  Continuing in this manner will yield an affine Weyl group element whose finite part has minimal length inside its conjugacy class. It is thus necessary to show that this inductive process terminates when we have an element for which non-emptiness is known.  The point of the next proposition will thus be to show that Proposition \ref{T:elliptic} applies to the resulting ``terminal'' element.

\begin{prop}[Geck-Pfeiffer]\label{T:min}
Let $w \in W$, and suppose that $\supp(w) = S$ and $w$ is of minimal length in its conjugacy class.  Then $w$ is elliptic.
\end{prop}

It is clear that if $w$ belongs to the Coxeter conjugacy class, then $w$ is elliptic.  If $G$ is of type $A_n$, being conjugate to a Coxeter element characterizes the elliptic elements (Proposition 3.1.16 in \cite{GPbook}).  In general, however, there are elliptic elements which are not conjugate to any Coxeter element, and Proposition \ref{T:min} appears as Proposition 3.1.12 in \cite{GPbook}. 

\end{subsection}

\begin{subsection}{Cyclic shift classes}

The proof of Theorem \ref{T:main} will proceed by induction on the length of the finite part of $x \in \widetilde{W}$.  Given $x \in \widetilde{W}$, each case of Proposition \ref{T:swsprop} provides a very specific element, say $x'$, to replace $x$ in the next step.  (In fact, in Cases (2) and (3), the element $x'$ provided is the unique element that provides both the requisite non-emptiness statement and the additivity hypothesis required to reapply the proposition.)  However, Proposition \ref{T:swsprop} may produce an element $x'=sxs$, for which the length of the finite part has not strictly decreased.  It is thus necessary to ensure that the inductive process does not terminate prematurely when conducted in this fashion; \textit{i.e.} that replacing $x$ by the prescribed $x'$ still enables us to eventually obtain an element whose finite part has minimal length in its conjugacy class.  To this extent, we require one additional result from \cite{GPbook} regarding cyclic shift classes.  We formulate the definition of the cyclic shift class that is most convenient for our purposes, although we remark that there are several other equivalent definitions.

\begin{defn} Let $w \in W$.
\begin{enumerate}
\item We say that $w' \in W$ is \emph{conjugate to $w$ by cyclic shift} if both $w \rightarrow w'$ and $w' \rightarrow w$.
\item By $\text{Cyc}(w)$ we denote the \emph{cyclic shift class of $w$}, which is the set of all elements $w' \in W$ that are conjugate to $w$ by cyclic shift.
\item A cyclic shift class $\text{Cyc}(w)$ is called \emph{terminal} if $w \rightarrow w'$ implies that $w' \in \text{Cyc}(w)$ for all $w' \in W$.
\end{enumerate}
\end{defn}

Observe that if $w \rightarrow w'$, then either $w' \in \text{Cyc}(w)$ or $\ell(w') < \ell(w)$.  Further, the above definition of $\text{Cyc}(w)$ also implies that \begin{equation}\text{Cyc}(w) = \{ w' \in W \mid w \rightarrow w'\ \text{and}\ \ell(w)  = \ell(w')\}.\end{equation}  It turns out that terminal cyclic shift classes of $w$ are subsets of the elements of minimal length inside the conjugacy class of $w$.  We will be interested in the case in which $w$ is also of full support, and the following corollary of the above results of Geck and Pfeiffer says that such elements are elliptic.

\begin{cor}\label{T:terminal}
If $\supp(w) = S$ and $\text{Cyc}(w)$ is terminal, then $w$ is elliptic.
\end{cor}

\noindent The main point of this result for our purposes is to say that we cannot be stuck in a situation in which both $w$ is not elliptic and the only cases of Proposition \ref{T:swsprop} that apply are (2a) or (3a).

\end{subsection}

\end{section}

\begin{section}{Exceptional Coxeter Elements}\label{S:except} 

Recall that in the statement of the main theorem, we impose the condition that $x\in \widetilde{W}$ be full; \textit{i.e.}, that the finite part of $x$ has full support.  One final obstruction to proving Theorem \ref{T:main} using an inductive argument would thus be if Proposition \ref{T:swsprop} never yielded a replacement element which was also full.  As we shall argue in the proof of Proposition \ref{T:induction}, it is clear in Cases (1), (2a), and (3a) of Proposition \ref{T:swsprop} that the replacement element $x'$ is again full.  By contrast, in Cases (2b) and (3b), the element $x'$ may or may not be full.  Our final goal is thus to analyze the elements $x$ such that we are both forced apply either Case (2b) or (3b) of Proposition \ref{T:swsprop}, and for which the resulting element $x'$ is never full.

\begin{subsection}{The fullness condition}\label{S:sxReuman}

The goal of this section will be to understand the finite Weyl group elements that can arise in the situation in which we must apply either Case (2b) or (3b) of Proposition \ref{T:swsprop}, but for which $sx$ or $xs$, respectively, is not full.  The next lemma provides a first approximation to understanding these elements, to which our inductive argument does not apply.  Fortunately, however, in this case the elements in consideration will turn out to be Coxeter of a certain form; namely, $w$ is the product of a simple reflection with Coxeter elements in two disjoint proper parabolic subgroups.  

\begin{lemma}\label{T:s_1unique}
Suppose that $G$ is of type $A_n$, and label the simple reflections $s_1, s_2, \dots, s_n$ such that $s_i$ interchanges the $i$ and $i+1$ coordinates.  Let $w \in W$, and suppose $\supp(w) = S$.  If $D_L(w) = \{s_j\}$ and $\supp(s_jw) \neq S$, then $w=s_js_{j-1} \cdots s_1 s_{j+1} \cdots s_n$.  Similarly, if $D_R(w) = \{s_k\}$ and $\supp(ws_k)\neq S$, then $w = s_1 \cdots s_{k-1}s_n \cdots s_{k+1}s_{k}$. In particular, $w$ is Coxeter in either case.
\end{lemma}

\begin{proof}
Consider a reduced expression for $w = s_{i_1} \cdots s_{i_q}$.  Since $\supp(s_jw)\neq S$, then $s_{i_1}$ is the only occurrence of $s_j$ in any reduced expression for $w$.  Since every element in the set $T_1:= \{ s_1, s_2, \dots, s_{j-1}\}$ commutes with every element in the set $T_2:= \{ s_{j+1}, \dots, s_n\}$, we can without loss of generality assume that $s_{i_2}, \dots, s_{i_{\ell}} \in T_1$ and $s_{i_{\ell+1}}, \dots, s_{i_q} \in T_2$ where $2 \leq \ell \leq q$.  Further, since $s_j$ is the unique left descent of $w$, there are only two choices for $s_{i_2}$ and $s_{i_{\ell+1}}$, namely $s_{i_2} = s_{j-1}$ and $s_{i_{\ell+1}}= s_{j+1} $.  Indeed, otherwise the left descent set for $w$ would also include $s_{i_2}$ or $s_{i_{\ell+1}}$, since these elements would commute with $s_j$ if $\{i_2, i_{\ell+1}\} \neq \{j \pm 1\}$.  Also observe that each of the products $s_{i_2}\cdots s_{i_\ell}$ and $s_{i_{\ell+1}} \cdots s_{i_k}$ have full support in the parabolic subgroups $W_{T_1}$ and $W_{T_2}$, respectively, since $w$ itself is of full support in $W$.

We now show that any element in the group $S_{m+1}$ for $m+1<n$ that is both of full support and for which every reduced expression begins with $s_m$ (resp. $s_1$) is of the form $s_m s_{m-1} \cdots s_1$ (resp. $s_1 s_2 \cdots s_m$)  in $S_{m+1}$.  Consider a permutation $\sigma \in S_{m+1}$ that is of full support, and suppose that $s_m$ is the unique left descent for $\sigma$.  Writing $\sigma = [a_1\ a_2\ \cdots \ a_{m+1}]$ in one-line notation, and writing $m =: a_i < a_j := m+1$, we must have $i > j$, but $a_k < a_{\ell}$ for all other pairs $k < \ell$.  However, in order for $\sigma$ to be of full support, we see that $i = m+1$ and $j = 1$, and the corresponding reduced expression is $\sigma = s_ms_{m-1}\cdots s_2s_1$.  An identical argument shows that if $s_1$ is the unique left descent for $\sigma$, then we must have $\sigma = s_1s_2\cdots s_m$.  

Applying this observation, we must have that $s_{i_2}\cdots s_{i_\ell} = s_{j-1}\cdots s_1 \in W_{T_1}$ and $s_{i_{\ell+1}} \cdots s_{i_k} = s_{j+1} \cdots s_n \in W_{T_2}$. Finally, since $\supp(s_jw) \neq S$, in which case $s_j$ is in neither $T_1$ nor $T_2$, we see that $w$ is a Coxeter element in $S_n$ of the form $w =s_js_{j-1} \cdots s_1 s_{j+1} \cdots s_n $.  

Applying the result for when $w$ has a unique left descent to $w^{-1}$ proves the claim in the case in which $D_R(w) =\{ s_k \}$ and $\supp(ws_{k}) \neq S$.
\end{proof}

\begin{remark}
We should point out that Lemma \ref{T:s_1unique} does not generalize to other types.  For example, let $G$ be of type $B_4$ and label the simple reflections $s_1, s_2, s_3, s_4$ so that $(s_1s_2)^4 = 1$.  Then $w = s_{43212}$ has the properties that $w$ is of full support, $D_L(w) = \{ s_4 \}$, and $\supp(s_4w) \neq S$.  However, $w$ is clearly not Coxeter, so the conclusion of Lemma \ref{T:s_1unique} fails in general for other types.

On the other hand, it is possible to characterize the elements that satisfy the hypotheses of Lemma \ref{T:s_1unique} in general.  Such elements are products of two \textit{fully commutative} elements in disjoint proper parabolic subgroups.  An element is fully commutative if any reduced expression can be obtained from any other by means of braid relations that only involve commuting generators.  We make further remarks in this direction following the proof of Proposition \ref{T:xCoxeter}.
\end{remark}

The next lemma is a generalization of Lemma \ref{T:s_1unique} and characterizes the structure of elements for which every application of Case (2b) of Proposition \ref{T:swsprop} yields an element that is not subsequently of full support.  We characterize elements in $W$ which have full support, but lose their full support when multiplied on the left by any simple reflection in the left descent set. We point out that this characterization is type-free.

\begin{lemma}\label{T:wform}
Let $w \in W$ be such that $\supp(w) = S$. Then $\supp(sw) \neq S$ for all $s \in  D_L(w)$ if and only $w$ can be written as a product $w = w_{\ell}w_1\cdots w_k$ such that
\begin{enumerate}
\item[($i$)] $w_{\ell}$ is a product of the elements in $D_L(w)$, all of which pairwise commute, 
\item[($ii$)] $k$ is the number of connected components of $S \ba D_L(w)$, and
\item[($iii$)] for $i = 1, \dots, k$, the element $w_i$ satisfies that $w_i \in W_{T_i}$ for $T_i \subsetneq S$, where
    \begin{enumerate} 
    \item[($a$)] each $T_i$ is a connected component of $S\ba D_L(w)$,
    \item[($b$)] the sets $D_L(w)$ and $T_i$ partition $S$, and
    \item[($c$)] $w_i$ has full support in $W_{T_i}$.
    \end{enumerate}
\end{enumerate}
\end{lemma}  

\begin{proof}
The reader can easily verify that any element $w \in W$ of full support that can be written as a product of the form described above satisfies that $\supp(sw) \neq S$ for all $s \in S$ such that $s \in D_L(w)$.

To prove the converse, we begin by arguing that the elements of $D_L(w)$ all commute with one another.  Under our hypotheses, $$D_L(w) =  \left\{ s \in S \mid \ell(sw) < \ell(w)\ \text{and} \ \supp(sw) = S-\{s\}\  \right\}.$$ Consider any two distinct elements $s, t \in D_L(w)$ so that $\ell(sw)<\ell(w)$ and $\ell(tw) < \ell(w)$.  It suffices to prove that there exists a reduced expression for $w$ beginning with the word $sts$, which would contradict the fact that $s \notin \supp(sw)$, unless $s$ and $t$ commute so that $sts=t$.  

To this extent, let $J = \{s,t\}$.  We prove more generally that $\ell(vw) = \ell(w) - \ell(v)$ for any $v \in W_J$.  Recall that for any $v,w \in W$, we have $\ell(vw) \geq \vert \ell(w) -\ell(v) \vert$, which yields $$\ell(vw) \geq \ell(w) - \ell(v).$$  Now observe that our hypotheses on $w, s,$ and $t$ imply that $w$ is the right coset representative of $W_J$ of maximal length.  For the other inequality, we use induction on $\ell(w) - \ell(v)$.  The base case is for $w = v$, in which the result follows trivially.  Now consider any $v < w$, where $<$ denotes the Bruhat order on $W^J$.  Then since $w$ has maximal length in $W^J$, either $v < sv$ or $v < tv$.  Without loss of generality, suppose that $v < sv$.  Then we see that \begin{align*}
\ell(vw) & \leq \ell(svw) + 1 \\ \phantom{\ell(vw)} & \leq \ell(w) - \ell(sv) +1\\  \phantom{\ell(vw)} &= \ell(w) - (\ell(v) + 1) + 1 \\  \phantom{\ell(vw)} &= \ell(w) - \ell(v),
\end{align*}
where we have used the induction hypothesis on $sv$ to obtain the second inequality.  Hence, we have that $\ell(vw) = \ell(w) - \ell(v)$ for any $v \in W_J$, and so there exists a reduced expression for $w$ beginning with the word $sts \in W_J$.

Further, observe that the above argument applies to say that there is a reduced expression for $w$ which begins with a product of all of the elements in $D_L(w)$.  This proves that we may write $w = w_{\ell}v$ as the product of two reduced words in $W$, where $w_{\ell}$ is a product of all of the elements in $D_L(w)$, which proves ($i$).

Since for every $s \in D_L(w)$ we have $s \notin \supp(sw)$, each $s \in D_L(w)$ only occurs in $w_{\ell}$ and not in $v$.  Therefore, we may also write $v = w_1 \cdots w_k$, where $w_i \in W_{T_i}$ and the $T_i\subsetneq S$ are all distinct connected components of $S\ba D_L(w)$, proving ($ii$).  Further, the property that $\supp(w)=S$ guarantees that we may further assume both that the sets $D_L(w)$ and $T_i$ form a partition of $S$ and that $\supp(w_i) = T_i$ for all $i=1, \dots, k$, which proves ($iii$) and completes the proof.
\end{proof}

By applying Lemma \ref{T:wform} to $w^{-1}$, one obtains an analogous result for right descents, which would characterize the elements for which every application of Case (3b) of Proposition \ref{T:swsprop} results in an element which does not have full support.

\begin{prop}\label{T:xCoxeter}
Suppose that $G$ is of type $A_n$.  Let $w \in W$ be of full support.  If $\supp(sw) \neq S$ for every $s
\in D_L(w)$ and $\supp(ws) \neq S$ for every $s \in D_R(w)$, then $w$ is Coxeter.
\end{prop}

\begin{proof}
Apply Lemma \ref{T:wform} to write $w = w_{\ell}w_1 \cdots w_k$ as the product of reduced words in proper parabolics satisfying the properties enumerated in the lemma.  If $\# D_R(w_i) = 1$ for all $i \in \{1, \dots, k \}$, then applying Lemma \ref{T:s_1unique} to $w_i$ says that each $w_i$ is Coxeter in $W_{T_i}$, and the conclusion follows.

Now suppose for a contradiction that there exists an $m \in \{1, \dots, k\}$ such that $\# D_R(w_m) \geq 2$.  If we label the simple reflections such that $s_j$ interchanges the $j$ and $j+1$ coordinates, then we may assume that each $T_i$ consists of consecutive simple reflections, and so we may write $T_m = \{s_{j_m}, \dots, s_{j_m+p_m-1}\}$.

We now claim that for all $s \in D_L(w_m^{-1})$, we have $\supp (sw_m^{-1}) \neq T_m$.  To this extent, consider some $s \in D_R(w_m) \subseteq T_m$. Then $s \in D_R(w)$ because all simple reflections in $T_m$ commute with those in $T_{\ell}$ for $\ell > m$, since $T_m$ and $T_{\ell}$ are disjoint connected components of $S\ba D_L(w)$. Hence, since $\supp(ws) \neq S$, then $\supp(w_ms) \neq T_m$ as well.  But because $\supp(w_m) = T_m$, this claim means that $w_m^{-1}$ satisfies the hypotheses of Lemma \ref{T:wform} with $W = W_{T_m}$ and $S=T_m$.  Therefore, Lemma \ref{T:wform} says that we may write $w_m = v_1\cdots v_qw_r$ as a product of reduced words in $W_{T_m}$ satisfying the properties listed in Lemma \ref{T:wform}.  In particular, $v_i \in W_{T'_i}$ where the $T'_i$ are all disjoint subsets of $T_m$.  Because  $\# D_R(w_m) \geq 2$, there exists a $t \in \{1,\dots, q\}$ such that $D_L(v_t)$ contains a simple reflection other than $s_{j_m}$ or $s_{j_m+p_m-1}$, say $s_t$.  Observe that then $s_t \in D_L(w_m)$ and that $s_t$ commutes with $w_i$ for all $i <m$, since all $w_i$ pairwise commute by Lemma \ref{T:wform}.  However, since $s_t \neq s_{j_m}, s_{j_m+p_m-1}$, we thus see that $s_t$ commutes with every element in $D_L(w)$, and that actually $s_t \in D_L(w)$, contradicting the fact that $T_m \cap D_L(w) = \emptyset$.
\end{proof}

Proposition \ref{T:xCoxeter} will serve as the base case in an inductive proof in of the main result in Section \ref{S:commdescent}, which completes the characterization of the elements to which Proposition \ref{T:swsprop} does not apply.

\begin{remark}
As in the case in Lemma \ref{T:s_1unique}, it is possible to prove that any element which satisfies the hypotheses of Proposition \ref{T:xCoxeter} for any $W$ is fully commutative.  Stembridge provides an extensive study of fully commutative elements in arbitrary Coxeter groups in \cite{StFC}, and we remark that it seems as though the combinatorial techniques in \cite{StFC} may provide a means for generalizing Proposition \ref{T:xCoxeter} to other types.  However, due to the fact that no type-free generalization for Proposition \ref{T:w0orCoxeter} is presently known, we do not pursue a generalization of Proposition \ref{T:xCoxeter} at this time.  In particular, it is possible that a generalization of Proposition \ref{T:w0orCoxeter} may not require a type-free version of Proposition \ref{T:xCoxeter}. 
\end{remark}

\end{subsection}

\begin{subsection}{Elements with Commuting Descents}\label{S:commdescent}

So far, we have yet to discuss the situation in which there exist simple reflections that commute with $w$, a situation which is excluded by Proposition \ref{T:swsprop}.  Our eventual claim will be that such elements are Coxeter in the case of $G= GL_n$, and to prove this result we will need an additional series of combinatorial lemmas.  

\begin{defn}
Let $w \in W$, and let $s \in D(w)$.  If $sw=ws$, we say that $s$ is a \emph{commuting descent for $w$}.  Observe that if $s$ is a commuting descent for $w$, then $s \in D_L(w) \cap D_R(w)$.
\end{defn}

To characterize the elements to which Proposition \ref{T:swsprop} does not apply, we also need to understand the elements of full support in $W$ which have commuting descents.

\begin{lemma}\label{T:w'full}
Let $w \in W$ be of full support, and let $s$ be a commuting descent for $w$. Write $w = sw'$, where $\ell(w') = \ell(w) - 1.$  Then $w'$ also has full support.
\end{lemma}

\begin{proof}
Suppose that $\supp(w') = S-\{s\}$.  Consider a reduced expression for $w'=s_{i_1}\cdots s_{i_k}$, in which we must have that $s\neq s_{i_j}$ for any $j$. In particular, all simple reflections with which $s$ does not commute must occur to the right of $s$ in the reduced expression $w=ss_{i_1}\cdots s_{i_k}$.  Therefore, no non-commuting braid relations involving $s$ can be applied to either $w$ or any word obtained from $w$ by applying braid relations.  On the other hand, there exists a reduced expression for $w$ of the form $w=w''s$, since $s$ is a commuting descent for $w$.  This means that $s$ must commute with all simple reflections in $w'$, which contradicts the hypothesis that $w$ has full support.
\end{proof}

\begin{lemma}\label{T:commuting}
Assume that $W = S_n$.  Let $w \in W$ be of full support, and suppose that $s$ and $t$ are commuting descents for $w$.  Then $s$ and $t$ commute with each other.
\end{lemma}

\begin{proof}
Label the simple reflections $s_1, s_2, \dots, s_{n-1}$ so that $s_i$ interchanges the $i$ and $i+1$ coordinates. Writing $w=[a_1\ a_2\  \cdots\  a_n]$ in one-line notation, multiplication on the right by $s_i$ interchanges $a_i$ and $a_{i+1}$.  Multiplication on the left by $s_i$ interchanges $a_j=i$ and $a_k=i+1$.  Therefore, if $s_i$ is a commuting descent for $w$, this means that multiplication on the left and the right by $s_i$ yield the same permutation.  In particular, we have that $a_i = i+1$ and $a_{i+1} = i$.  Looking at the disjoint cycle decomposition for $w$, we must then have $$w = (i\ \ i+1)(a_{11}\ a_{12} \ \cdots\ a_{1k_1})\cdots (a_{j1}\ a_{j2}\ \cdots\ a_{jk_j}).$$ By the same argument, if $s_j$ is any other commuting descent for $w$, the disjoint cycle decomposition for $w$ contains the transposition $(j\ \ j+1)$, which is distinct from the transposition $(i\ \ i+1)$.  The cycles $(i\ \ i+1)$ and $(j\ \ j+1)$ therefore commute in the disjoint cycle representation of $w$, which says precisely that $s_i$ and $s_j$ commute.
\end{proof}

\begin{remark}
We point out that Lemma \ref{T:commuting} clearly fails in other types.  For example, in type $B_n$, every simple reflection is a commuting descent of the longest element $w_0$.  The conclusion that all simple reflections pairwise commute is absurd, however.
\end{remark}

\begin{lemma}\label{T:commutingCoxeter}
Let $w \in W$, where $W$ is reduced and does not contain a factor of type $A_1$.  If $w$ is Coxeter, then the set of commuting descents for $w$ is empty.
\end{lemma}

\begin{proof}
Let $w = s_{i_1} \cdots s_{i_n}$ be a Coxeter element.  Then $s_jw = ws_j$ if and only if $w(\alpha_j) = \pm \alpha_j$, where $\alpha_j$ is the simple root corresponding to $s_j$.  Suppose that $i_k = j$ in the reduced expression for $w$.  We argue that $s_{i_{k+1}}\cdots s_{i_n}(\alpha_j)\neq \pm s_{i_k}\cdots s_{i_1}(\alpha_j)$, in which case $s_j$ and $w$ do not commute.  Indeed, recall that $s_k(\beta) = \beta - \langle \beta, \alpha_k^{\vee}\rangle \alpha_k$ for a root $\beta$. We then write \begin{align} \theta &:= s_{i_{k+1}}\cdots s_{i_n}(\alpha_j) = \alpha_j + x_{i_{k+1}}\alpha_{i_{k+1}} + \cdots + x_{i_n}\alpha_{i_n} \\
 \theta' &:=  s_{i_{k-1}}\cdots s_{i_1}(\alpha_j) = \alpha_j + x_{i_1}\alpha_{i_1} + \cdots + x_{i_{k-1}}\alpha_{i_{k-1}},
\end{align} where at least one of the $x_{\ell} \neq 0$, since otherwise $s_j$ would commute with all other simple reflections, which is impossible for a reduced root system.  Now, applying $s_{i_k} = s_j$ to $\theta'$, we have \begin{equation} s_j (\theta') = y\alpha_j +  x_{i_1}\alpha_{i_1} + \cdots + x_{i_{k-1}}\alpha_{i_{k-1}}, \end{equation} for some $y \in \R$.  Then observe that if $\theta = \pm s_j\theta '$, then $x_{\ell} = 0$ for all $\ell$, which is a contradiction.
\end{proof}

\begin{prop}\label{T:w0orCoxeter}
Suppose that $G$ is of type $A_n$. Let $w \in W$ be of full support.   If
\begin{enumerate}
\item[($i$)] for every $s \in D_L(w)$, either $s \notin \supp(sw)$ or $sw = ws$, and
\item[($ii$)] for every $s \in D_R(w)$, either $s \notin \supp(ws)$ or $sw = ws$,
\end{enumerate}
then $w$ is Coxeter.
\end{prop}

\begin{proof}
Denote by $D_w^c = \{s \in D(w) \mid sw=ws \}$ the set of commuting descents for $w$.  We proceed by induction on the size of $D_w^c$.  If $D_w^c = \emptyset$, then Proposition \ref{T:xCoxeter} yields the result.  Now consider any $w \in W$ of full support satisfying both ($i$) and ($ii$), and suppose that $D_w^c \neq \emptyset$.  Consider $s_i \in D_w^c$, and rite $w = s_iw'$, where $\ell(w') = \ell(w) - 1.$  We first claim that $D_{w'}^c = D_w^c - \{s_i\}$.  Indeed, in the proof of Lemma \ref{T:commuting}, we showed that if $s_i$ is a commuting descent for $w$, then the disjoint cycle notation for $w$ contains the transposition $(i \ \ i+1)$. Multiplying $w$ by $s_i$ merely removes $(i\ \ i+1)$ from the disjoint cycle decomposition for $w$, leaving intact all transpositions $(j\ \ j+1)$ corresponding to any other $s_j \in D_w^c$.  Therefore, $D_{w'}^c = D_w^c - \{s_i\}$.

Lemma \ref{T:w'full} applies to say that $w'$ is of full support.  We now argue that $w'$ satisfies the other hypotheses of the proposition so that induction applies to $w'$. Consider $t \in D_L(w')$.  We wish to show that either $\supp(tw') = S-\{t\}$ or that $tw'=w't$.  Observe that since $w=s_iw'$, then $w'=s_iw=ws_i=s_iw's_i$ so that $w=s_iw'=w's_i$.  Consider a reduced expression for $w'=s_{j_1} \cdots s_{j_k}$.  Since $t \in D_L(w')$, we must have that $tw' = s_{j_1}\cdots \widehat{s_{j_{\ell}}} \cdots s_{j_k}$ for some $j_{\ell}$. But then, since $tw's_i = tw=ts_iw'$, by substitution we obtain $$s_{j_1} \cdots \widehat{s_{j_{\ell}}} \cdots s_{j_k} s_i = ts_is_{j_1} \cdots s_{j_k}.$$  Comparing the lengths of the elements on both sides of the above equality and recalling that $\ell(w)=k+1$, we see that necessarily have $t \in D_L(w)$ as well.

Since $t \in D_L(w)$, by hypothesis we know that either $t \notin \supp(tw)$ or $tw=wt$.  Suppose first that $t \notin \supp(tw)$.  Since $$tw = s_{j_1} \cdots \widehat{s_{j_{\ell}}} \cdots s_{j_k}s_i = tw's_i,$$ right multiplication by $s_i$ yields $tw' = s_{j_1}\cdots \widehat{s_{j_{\ell}}} \cdots s_{j_k}$.  Since $t \notin \supp(tw),$ and $tw'$ is a subword of $tw$, we see that $t \notin \supp(tw')$.  If, on the other hand, $t$ commutes with $w$, then $tw = tw's_i = wt = w's_it.$  But by Lemma \ref{T:commuting}, $s_i$ and $t$ commute with each other, since they are both commuting descents for $w$.  Therefore, we have $tw's_i = w's_it = w'ts_i$, from which we conclude that $tw'=w't$.  

Altogether, we have shown both that $w'$ is of full support, and if $t \in D_L(w')$, then either $t \notin \supp(tw')$ or $tw'=w't$.  An identical argument applies to conclude that if instead $t \in D_R(w')$, then either $t \notin \supp(w't)$ or $tw'=w't$.  Since $D_{w'}^c = D_w^c - \{s_i\}$, induction applies to $w'$, which means that $w'$ is Coxeter.  On the other hand, by construction we have that $s_iw'=w's_i$, contradicting Lemma \ref{T:commutingCoxeter}.  Therefore, we conclude that in fact $D_w^c \neq \emptyset$ can never hold.  We thus have demonstrated that $D_w^c = \emptyset$, and Proposition \ref{T:xCoxeter} therefore applies directly to $w$, from which we conclude that $w$ itself must be Coxeter. 
\end{proof}

\end{subsection}
\end{section}

\begin{section}{Proof of the main theorem}\label{S:thmproof}

\begin{subsection}{Proof of Theorem \ref{T:main}}
Throughout this section, we fix a basic element $b \in G(L)$.  Because we will use Proposition \ref{T:w0orCoxeter}, we must also assume that $G$ is of type $A_n$.  We begin by stating a proposition that combines the main results of Sections \ref{S:length}, \ref{S:conj}, and \ref{S:except} to establish the inductive framework for the proof of the main theorem.

\begin{prop}\label{T:induction}
Let $G$ be of type $A_n$.  Suppose that $x = \pi^{v(\mu)}w \in \widetilde{W}$ is full, additive, and $\mu$ is regular dominant.  Then either $x$ is elliptic, or there exists an element $y = \pi^{v'(\mu)}w' \in \widetilde{W}$ such that the following all hold:
\begin{enumerate}
\item[(1)] If $X_y(b) \neq \emptyset$, then $X_x(b) \neq \emptyset$;
\item[(2)] $y$ is additive;
\item[(3)] $y$ is full; and
\item[(4)] $\ell(w') < \ell(w)$.
\end{enumerate}
\end{prop}

\begin{proof}
Assume that $x$ is not elliptic.  We must prove the existence of an element $y \in \widetilde{W}$ which satisfies the above properties.  Since $x$ is full so that $\supp(w) = S$, then Proposition \ref{T:min} says that $w$ does not have minimal length inside its conjugacy class.  Define the following set $$S_w:= \{ s \in D(w) \mid \ell(sws) \leq \ell(w)\ \text{and}\ w \neq sws\}.$$  Observe that since $w$ is not of minimal length in its conjugacy class, then there must exist an $s \in D(w)$ such that $\ell(sws) \leq \ell(w)$.  Furthermore, there must exist an $s \in D(w)$ such that $w \neq sws$, since otherwise $\text{Cyc}(w) = \{ w \}$ is terminal, which contradicts Corollary \ref{T:terminal}. Therefore, the set $S_w$ is non-empty.  Finally, $x$ is additive with $\mu$ regular dominant by assumption, and so the hypotheses of Proposition \ref{T:swsprop} are satisfied by $x$ and any $s \in S_w$.  Applying Proposition \ref{T:swsprop}, we obtain an element $y_s \in \{sxs, sx, xs \}$ which automatically satisfies both properties (1) and (2) in the statement of this proposition. 

Now suppose that for every $s \in S_w$, the element $y_s$ which results from applying Proposition \ref{T:swsprop} is not full.  In Case (1) of Proposition \ref{T:swsprop}, the element $w' = sw$ satisfies $\supp(w') = S$, since $\ell(sws) = \ell(w)-2$, which means that $y_s$ is full in this case.  Similarly, in Case (2a) and Case (3a), the element $w' = sws$ again has full support.  Hence, if $y_s$ is never full, this means that only Cases (2b) or (3b) of Proposition \ref{T:swsprop} apply to $x$.  In particular, for any $s \in S_w$, we have either $\supp(sw) \neq S$ or $\supp(ws) \neq S$.  In this situation, we see that for every $s \in D_L(w)$ we have either $s \notin \supp(sw)$ or $sw=ws$, and for every $s \in D_R(w)$ we have either $s \notin \supp(ws)$ or $sw=ws$.  Proposition \ref{T:w0orCoxeter} then applies to say that $w$ is Coxeter, in which case $x$ is elliptic, yielding a contradiction.  Therefore, there exists an element $s \in S_w$ such that the resulting $y_s$ is again full, satisfying property (3).  

Recall that $\eta_1: \widetilde{W} \rightarrow W$ is the surjection which isolates the finite part of an affine Weyl group element.  We define a subset of $S_w$ consisting of elements such that the corresponding element $y_s$ obtained by applying Proposition \ref{T:swsprop} has finite part with full support: $$S^f_w := \{ s \in S_w \mid \supp(\eta_1(y_s)) = S \},$$ which we have just proved to be non-empty.  If there exists an $s \in S^f_w$ such that $\ell(\eta_1(y_s)) < \ell(w)$, then $y_s$ also satisfies property (4), and we are done.  

Consider the case in which for every $s \in S^f_w$, we have $\ell(\eta_1(y_s)) = \ell(w)$.  This means that for every $s \in S^f_w$, either Case (2a) or (3a) of Proposition \ref{T:swsprop} applies.  We first argue that, in this situation, $S^f_w = D(w)$.  If we may only apply Case (2a) or (3a) of Proposition \ref{T:swsprop} to $x$, then for any $s \in D(w)$, either $w \neq sws$ and $\ell(sws) = \ell(w)$, or $w = sws$.  Now, if $w = sws$ for some $s \in D(w)$, then $s \in D_L(w) \cap D_R(w)$. For any $s \in D_L(w) \cap D_R(w)$ we have $\ell(v^{-1}ws) = \ell(v^{-1}w) - 1$, which means that we are in Case (2b) of Proposition \ref{T:swsprop}, a contradiction.  Therefore, $w \neq sws$ for all $s \in D(w)$, and $D(w) = S_w$.  However, if only Case (2a) or (3a) applies to $x$, then for any $s \in D(w)$, we have $\eta_1(y_s) = sws$ so that $y_s$ is always full, proving that $D(w) = S^f_w$.  Therefore, any choice of $s \in D(w)$ yields an element $y_s$ satisfying properties (1) - (3) of the claim.  In particular, $y_s$ is both additive and full.  Also note that $y_s$ is not elliptic, since $x$ was not.  Therefore, we may apply the previous argument to $y_s$ instead of $x$ to obtain an element $y_t$ satisfying properties (1) - (3) of the claim.  If there exists a $t \in S^f_{sws}$ such that $\ell(\eta_1(y_t)) = \ell(tswst) <\ell(sws) = \ell(w)$, then property (4) is satisfied, and we are done.  

Continuing in this manner, it remains only to discuss the situation in which repeated application of this argument never results in an element $y \in \widetilde{W}$ such that $\ell(\eta_1(y)) < \ell(w)$.  Consider any $w' \in W$ such that $w \rightarrow w'$.  By definition, there exists a sequence of simple reflections $s_i \in S$ such that $w = w_0 \xrightarrow{s_1} w_1 \xrightarrow{s_2} \cdots \xrightarrow{s_k} w_k = w'$, where $w_i = s_iw_{i-1}s_i$ and $\ell(w_i) \leq \ell(w_{i-1})$ for all $1 \leq i \leq k$.  Such a sequence $\{ s_i \}$ corresponds to a sequence of applications of Proposition \ref{T:swsprop} Cases (2a) or (3a), depending on whether $s_{i+1} \in D_L(w_i)$ or $s_{i+1} \in D_R(w_i)$ respectively, since $S^f_{w_i} = D(w_i)$ for all $0 \leq i \leq k$.  Therefore, if $\ell(w_i) = \ell(w)$ for all $1 \leq i \leq k$, then $w' \in \text{Cyc}(w)$, which means that $\text{Cyc}(w)$ is terminal, contradicting Corollary \ref{T:terminal}.  Therefore, after a finite number of applications of Cases (2a) and (3a), Proposition \ref{T:swsprop} must yield an affine Weyl group element whose finite part decreases in length, in addition to satisfying properties (1) - (3), and such an element will satisfy all four properties in the proposition.
\end{proof}

\begin{proof}[Proof of Theorem \ref{T:main}]

Suppose that $G$ is of type $A_n$.  Let $b \in G(L)$ be basic, and let $x = \pi^{v(\mu)}w$, where $\mu$ is regular dominant and $\ell(v^{-1}w) = \ell(v^{-1}) + \ell(w)$.  Further suppose that $\supp(w) = S$.

We proceed by induction on $\ell(w)$.  The base case occurs when $\ell(w) = n$ and $w$ is Coxeter, in which case $x$ is elliptic and the non-emptiness result follows by Proposition \ref{T:elliptic}.  Now suppose $\ell(w) > n$.  If $x$ is elliptic, Proposition \ref{T:elliptic} yield the result.  Otherwise, Proposition \ref{T:induction} applies to say that there exists an affine Weyl group element $y$ such that $X_y(b) \neq \emptyset$ implies that $X_x(b) \neq \emptyset$, where $y$ is additive and full, and the finite part of $y$ has length smaller than $\ell(w)$.  The result thus follows for $y$ by induction, which yields the non-emptiness result for $x$ as a consequence.

The reader may easily check the result for types $C_2$ and $G_2$ by hand.
\end{proof}

\end{subsection}

\end{section}

\bibliographystyle{amsplain}
\bibliography{references}

\end{document}